\documentclass[12pt,english]{amsart}
\usepackage[cp1251]{inputenc}
\usepackage[english]{babel}
\usepackage{amsmath,amsthm,amssymb}
\newtheorem{theorem}{Theorem}[section]
\newtheorem{lemma}[theorem]{Lemma}
\newtheorem{corollary}[theorem]{Corollary}
\newtheorem{remark}[theorem]{Remark}

\theoremstyle{definition}

\numberwithin{equation}{section}

\begin{document}

\title[]
{Bounds for the total variation distance between
second degree polynomials in normal random variables}

\author{Egor Kosov}

\maketitle

\begin{abstract}
In this paper we study
bounds for the total variation
distance between two second degree polynomials in
normal random variables provided that
they essentially depend on at least three variables.
\end{abstract}

\noindent
Keywords: distribution of a polynomial, total variation distance, Gaussian measure

\noindent
AMS Subject Classification: 60E05, 	60E15, 60B10, 28C20

\section{Introduction}

Let $Z=(Z_1,\ldots, Z_n)$ be the standard normal random vector in $\mathbb{R}^n$,
i.e. $Z_k$ --- are i.i.d normal random variables with the distribution $\mathcal{N}(0, 1)$.
We consider two polynomials $f$ and $g$ of degree at most $d$ in $n$ variables
and we are interested in the bounds for the total variation distance
$d_{\rm TV}\bigl(f(Z), g(Z)\bigr)$ between the random variables $f(Z)$ and $g(Z)$,
where
the total variation distance is defined as follows:
$$
d_{\rm TV}\bigl(f(Z), g(Z)\bigr) := \sup\biggl\{
\mathbb{E}\bigl[\varphi\bigl(f(Z)\bigr)\bigr] - \mathbb{E}\bigl[\varphi\bigl(g(Z)\bigr)\bigr],
\ \varphi\in C_0^\infty(\mathbb{R}),\ \|\varphi\|_\infty\le1\biggr\},
$$
where $\|\varphi\|_\infty:=\sup\limits_{t\in \mathbb{R}}|\varphi(t)|$.

The Davydov--Martynova
bound (see \cite{DavM})
asserts that, for each $d$ and for each non-constant polynomial $g$ of degree at most $d$,
there is a number
$C(d, g)$ such that for each polynomial $f$ of degree at most $d$ one has
\begin{equation}\label{DM-eq}
d_{\rm TV}\bigl(f(Z), g(Z)\bigr)\le C(d, g) \|f(Z)-g(Z)\|_2^{1/d},
\end{equation}
where we use the notation $\|f(Z)-g(Z)\|_2:=\bigl(\mathbb{E}|f(Z)-g(Z)|^2\bigr)^{1/2}$
for short.
In \cite{DavM} the bound was stated only for Hermite polynomials of a fixed
degree (for elements of a fixed Wiener chaos) but the estimate is actually valid for any
polynomials of a fixed degree (see also \cite{NP}, \cite{NNP},
\cite{Bret}, \cite{BKZ}, \cite{BKZ-DAN}, \cite{Zel}
for further development of the stated inequality,
where the case of random vectors with polynomial components has also been studied,
and see also two survey papers \cite{B16} and \cite{B19}).
In \cite{Kos-IMRN} (see also \cite{Kos-Dan}) the dependence of $C(d, g)$ on $g$ was clarified:
one can take $C(d, g)=C(d)\bigl(\mathbb{D}[g(Z)]\bigr)^{-\frac{1}{2d}}$,
where $\mathbb{D}[g(Z)]$ is the variance of the random variable $g(Z)$.
Moreover, in that paper the Davydov--Martynova inequality \eqref{DM-eq} was generalized
to the case when the random vector $Z$ has an arbitrary logarithmically concave
distribution (e.g. $Z$ has a uniform distribution on some convex body,
see \cite{Bor} or
\cite[Section~3.10(vi)]{MeasTheor} and \cite[Section~4.3]{DiffMeas}).
In recent years similar bounds have been studied in more general settings
of general sufficiently smooth functions $f$ and $g$
(e.g. for functions from Sobolev classes) and general spaces and distributions of random element $Z$
(see \cite{BC}, \cite{BC17}, \cite{BC19}, \cite{BCP}, \cite{Kos-FCAA}).
We point out that bounds of the Davydov--Martynova type have found their applications in statistics
(e.g. see \cite{DEKN}, \cite{EV}, \cite{OV}) as such bounds
provide a rate of convergence in total variation for a
sequence of random variables convergent in a weaker distance.

In \cite{NP}, along with the Davydov--Martynova-type inequality \eqref{DM-eq},
I. Nourdin and G. Poly studied a similar bound of the following type:
$$
d_{\rm TV}\bigl(f(Z), g(Z)\bigr)\le C\bigl(d, \mathbb{D}[f(Z)], \mathbb{D}[g(Z)]\bigr)\cdot
 \bigl(d_{\rm KR}(f(Z), g(Z))\bigr)^{\theta(d)},
$$
where $d_{\rm KR}$ is the Kantorovich--Rubinstein distance between distributions.
The best up-to-date result asserts that one can take the exponent $\theta(d)=\frac{1}{d+1}$
(see \cite{Kos-JMAA}). An interesting question is whether one can take the exponent
$\theta(d)=\frac{1}{d}$ like in estimate \eqref{DM-eq}.
The proof in \cite{Kos-JMAA}
shows that bounds for the total variation distance are connected with the smoothness
of the distributions of random variables $f(Z)$ and $g(Z)$
(see also \cite{BKZ}, \cite{BZ}, \cite{Kos-Sb}, \cite{Kos-DAN}, \cite{BKP-DAN}, and \cite{BKP-MMJ},
where the role of fractional regularity was studied and generalized).
It is known (e.g. see \cite[Corollary~6.9.9]{Gaus}) that the distribution of
quadratic forms and general second degree polynomials
becomes smoother
with the growth of the number
of variables they are essentially depend on.
Thus, it is natural to study possible improvements
in the
Davydov--Martynova-type inequality
when we deal with the second degree polynomials
that depend essentially on sufficiently many variables.
Note that every second degree polynomial $g$ can be represented in the form
$g(x)=\langle Bx, x\rangle+\langle b, x\rangle + \beta$
for some self adjoint linear operator $B$, for some vector $b\in \mathbb{R}^n$,
and for some $\beta\in \mathbb{R}$. For a self adjoint linear operator $B$,
we let $\Lambda(B):=\{\lambda_1, \lambda_2,\ldots \}$ to be the set of all eigenvalues
of $B$, counting multiplicities, and we always assume that the eigenvalues are enumerated
such that $|\lambda_1|\ge|\lambda_2|\ge|\lambda_3|\ge\ldots$.
In \cite{Zin} the following assertion was proved:
if $g(x) = \langle Bx, x\rangle-{\rm tr}B$
and if the cardinality of the set $\{\lambda\in \Lambda(B)\colon \lambda\ne 0\}$
is at least $5$, then there is a number $C(g)$ such that, for any
$f(x)=\langle Ax, x\rangle - {\rm tr}A$ (with self adjoint $A$),
one has
\begin{equation}\label{Zin-eq}
d_{\rm TV}\bigl(f(Z), g(Z)\bigr)\le C(g) \|f(Z)-g(Z)\|_2.
\end{equation}
In the present note we continue the study of the total variation distance bounds
for polynomials in Gaussian random variables and
are interested in possible improvements
of \eqref{Zin-eq}.
In particular, we show (see Corollary \ref{cor2}) that,
for polynomials $f(x)=\langle Ax, x\rangle+\langle a, x\rangle + \alpha$ and
$g(x)=\langle Bx, x\rangle+\langle b, x\rangle + \beta$ (with self adjoint $A$ and $B$),
one has
\begin{equation}\label{Kos-eq}
d_{\rm TV}\bigl(f(Z), g(Z)\bigr)
\le \frac{80}{\sqrt{|s_1|\cdot|s_2|}}\|f(Z)-g(Z)\|_2
\end{equation}
provided that the cardinality of
the set $\{\lambda\in \Lambda(B)\colon \lambda\ne 0\}$
is at least $3$,
where $s_1$ and $s_2$ are any eigenvalues of the same sign
(such eigenvalues exist under this assumption).
In particular, under the same assumption,
$$
d_{\rm TV}\bigl(f(Z), g(Z)\bigr)
\le \frac{80}{\sqrt{|\lambda_2|\cdot|\lambda_3|}}\|f(Z)-g(Z)\|_2.
$$
We note that this bound is optimal in the following sense:
if $g(x_1, x_2) = x_1^2-x_2^2$ (i.e. only two eigenvalues are non-zero),
then the bound of type \eqref{Zin-eq} is not true.
Indeed, if the bound \eqref{Zin-eq} were valid for this $g$, then one could take
$f=g+h$, $h\in \mathbb{R}$ and get the bound
$$
d_{\rm TV}\bigl(g(Z)+h, g(Z)\bigr)\le C\cdot h.
$$
This estimate implies that $g(Z)$ has a bounded density which is known not to be the case.
The proof of \eqref{Kos-eq} follows the ideas developed
in \cite{BKP-IZV} and \cite{BKP-homog-DAN},
where regularity of densities of
strictly convex functions in normal random variables was studied.

In the present paper
we also study connections with and possible generalizations
of the recent results
concerning the
distributions of Euclidean norms of normal random vectors
(see \cite{GNSU}, \cite{NSU}, \cite{BNU}).
This distributions
have been extensively studied
due to their broad applications in statistics.
In particular, in \cite{GNSU} the following assertion was proved.
Let $X$ and $Y$ be any normal random vectors in $\mathbb{R}^n$ with
covariance matrixes $\Sigma_X$ and $\Sigma_Y$ respectively and let $a\in \mathbb{R}^n$.
Let
$\Lambda_{kX}^2:=\sum\limits_{j=k}^{\infty}\lambda_{jX}^2$,
where $\lambda_{jX}$ are the eigenvalues of the covariance matrix $\Sigma_X$ of the
random vector $X$, counting multiplicities and
arranged in the non-increasing order, and let $\Lambda_{kY}$
be defined in the same manner for the random vector $Y$.
Then
\begin{equation}\label{GNSU-eq}
d_{\rm Kol}(|X|, |Y-a|)\le
C\Bigl(\frac{1}{\sqrt{\Lambda_{1X}\Lambda_{2X}}}+\frac{1}{\sqrt{\Lambda_{1Y}\Lambda_{2Y}}}\Bigr)
\Bigl(\|\Sigma_X - \Sigma_Y\|_{(1)}+|a|^2\Bigr)
\end{equation}
for some numerical constant $C>0$
where $\|\cdot\|_{(1)}$ is the nuclear norm of a matrix and
where $d_{\rm Kol}$ is the Kolmogorov distance between random variables i.e.
$$
d_{\rm Kol}(\xi, \eta):=\sup\limits_{t\in \mathbb{R}}\bigl|P(\xi\le t) - P(\eta\le t)\bigr|.
$$
As a corollary one also has (see \cite[Corollary~2.3]{GNSU})
\begin{equation}\label{GNSU-eq-2}
d_{\rm Kol}(|X-a|, |Y-a|)\le
C\Bigl(\frac{1}{\sqrt{\Lambda_{1X}\Lambda_{2X}}}+\frac{1}{\sqrt{\Lambda_{1Y}\Lambda_{2Y}}}\Bigr)
\Bigl(\|\Sigma_X - \Sigma_Y\|_{(1)}+|a|^2\Bigr).
\end{equation}
In this note we provide some
generalizations of the bounds
\eqref{GNSU-eq} and \eqref{GNSU-eq-2}
replacing the Kolmogorov distance with the total variation
distance but at the cost of a worse constant
and a differen right hand side of the estimate.
In particular, we prove (see Corollary \ref{cor3}) that
\begin{multline}\label{Kos-eq-2}
d_{\rm TV}\bigl(|X-a|, |Y-b|\bigr)
\\
\le
\frac{160}{\sqrt{\lambda_{1X}\cdot\lambda_{2X}}}
\bigl(\|\Sigma_X-\Sigma_Y\|_{HS}+
|{\rm tr}\Sigma_X-{\rm tr}\Sigma_Y|+\bigl||a|^2-|b|^2\bigr|+
|\Sigma_X^{1/2}a - \Sigma_Y^{1/2}b|\bigr).
\end{multline}
Let us compare the obtained bound with estimates \eqref{GNSU-eq}
and \eqref{GNSU-eq-2}.
First of all we note that the constant
$\frac{1}{\sqrt{\lambda_{1X}\cdot\lambda_{2X}}}$
is worse (and, often, significantly) compared to
$\frac{1}{\sqrt{\Lambda_{1X}\Lambda_{2X}}}$.
It would be interesting to understand
if the constant $\frac{1}{\sqrt{\lambda_{1X}\cdot\lambda_{2X}}}$
could be replaced with $\frac{1}{\sqrt{\Lambda_{1X}\Lambda_{2X}}}$
in the bound for the total variation distance.
Now, let us compare 
the right hand side in the obtained bound when $b=0$
with $\|\Sigma_X - \Sigma_Y\|_{(1)}+|a|^2$ from \eqref{GNSU-eq}.
On the one hand, when the shift $a\ne 0$, 
the right hand side in \eqref{Kos-eq-2} provides
worse decay in $a$ when $a\to 0$.
Indeed, in the case when $X=Y=Z$, where
$Z$ is the standard $n$-dimensional normal random vector, $b=0$,
and vector $a$ has a small norm,
we have $|a|^2 = \bar{o}(|\Sigma_X^{1/2}a|\bigr)$.
On the other hand, when $b=a=0$ and ${\rm tr}\Sigma_X={\rm tr}\Sigma_Y$
(i.e. $\mathbb{E}|X|^2 = \mathbb{E}|Y|^2$),
on the right hand side of
\eqref{Kos-eq-2} we get the Hilbert--Schmidt norm
$\|\Sigma_X-\Sigma_Y\|_{HS}$
which is smaller than the nuclear norm $\|\Sigma_X-\Sigma_Y\|_{(1)}$
from the bound \eqref{GNSU-eq}.

\section{Proof of the main technical result}

The main result of this section is the following theorem.

\begin{theorem}\label{T1}
Let $Z$ be the standard $n$-dimensional normal random vector and
let $d\in \mathbb{N}$. There is a number $C(d)$, dependent only on $d$,
such that for any polynomial
$f$ of degree at most $d$ on $\mathbb{R}^n$ and for any
second degree polynomial
$g(x) = \langle Bx, x\rangle +\langle b, x\rangle +\beta$,
where $B$ is a self-adjoint operator,
$b\in \mathbb{R}^n$, and
$\beta\in \mathbb{R}$, one has
$$
d_{\rm TV}\bigl(f(Z), g(Z)\bigr)\le \frac{C(d)}{\sqrt{|s_1|\cdot|s_2|}}\|f(Z)-g(Z)\|_2
$$
provided that there is a two dimensional space $L$
such that the quadratic form $x\mapsto\langle Bx, x\rangle$
is positively or negatively defined on this subspace
and $s_1$ and $s_2$ are the eigenvalues of this quadratic form.
When $f$ is of degree at most $2$, one can take $C(2)=80$.
\end{theorem}

We firstly prove the following auxiliary lemma.

\begin{lemma}\label{lem1}
Let $f$ be any polynomial on $\mathbb{R}^2$
and let $g$ be a second degree polynomial on $\mathbb{R}^2$,
i.e. it has the form
$g(x) = \langle Bx, x\rangle +\langle b, x\rangle +\beta$
for some self adjoint $B$, $b\in \mathrm{\mathbb{R}^2}$, and $\beta\in \mathbb{R}$.
Assume that $B$ is positively or negatively defined.
Then for any $\varphi\in C_0^\infty(\mathbb{R})$, $\|\varphi\|_\infty\le 1$,
one has
\begin{multline*}
\mathbb{E}\bigl[\varphi\bigl(f(X)\bigr)\bigr] - \mathbb{E}\bigl[\varphi\bigl(g(X)\bigr)\bigr]
\\
\le
\frac{\sqrt{2\pi}}{\sqrt{|s_1|\cdot|s_2|}}
\Bigl[3\sup_{x\in \mathbb{R}^2}\bigl[|f(x)-g(x)|e^{-\frac{1}{4}|x|^2}\bigr]
+
\sup\limits_{x\in \mathbb{R}^2}\bigl[|\nabla f(x) - \nabla g(x)| e^{-\frac{1}{4}|x|^2}\bigr]\Bigr],
\end{multline*}
where $s_1$ and $s_2$ are the eigenvalues of $B$ and
where $X$ is the standard normal random vector in $\mathbb{R}^2$.
\end{lemma}

\begin{proof}
We can assume that $g(x) = \langle B(x-t), x-t\rangle+c$
where $t\in \mathbb{R}^2$,
$c\in \mathbb{R}$.
In polar coordinates with respect to the center $t$ we have:
\begin{multline*}
\int_{\mathbb{R}^2}[\varphi(f(x))-\varphi(g(x))]\, \gamma(dx)
\\
=
(2\pi)^{-1}\int_0^{2\pi}\int_{0}^{+\infty}
[\varphi(f(t+r\hat{\theta})) - \varphi(g(t+r\hat{\theta}))]
re^{-\frac{1}{2}|t+r\hat{\theta}|^2}\, dr\, d\theta,
\end{multline*}
where $\hat{\theta}:=(\cos\theta, \sin\theta)$.
For a fixed $\theta\in [0, 2\pi)$
we set $f_\theta(r):=f(t+r\hat{\theta})$, $g_\theta(r):=g(t+r\hat{\theta})$, and
$\varrho_\theta(r) = (2\pi)^{-1}e^{-\frac{1}{2}|t+r\hat{\theta}|^2}$.
Let
$$
\Phi(t) = \int_{-\infty}^t\varphi(\tau)d\tau,
$$
then
$$
\partial_r(\Phi(f_\theta) - \Phi(g_\theta)) =
\partial_r f_\theta\cdot\varphi(f_\theta)-\partial_r g_\theta\cdot \varphi(g_\theta)
=
\bigl(\varphi(f_\theta)-\varphi(g_\theta)\bigr)\partial_r g_\theta+
\varphi(f_\theta)\bigl(\partial_r f_\theta- \partial_r g_\theta\bigr).
$$
We note that $\partial_r g_\theta(r) = 2r\langle B\hat{\theta}, \hat{\theta}\rangle$.
Thus,
\begin{multline*}
\int_{0}^{+\infty}
[\varphi(f_\theta) - \varphi(g_\theta)]
r\varrho_\theta\, dr
=
\frac{1}{2\langle B\hat{\theta},\hat{\theta}\rangle}\int_{0}^{+\infty}
[\varphi(f_\theta) - \varphi(g_\theta)]\partial_rg_\theta
\varrho_\theta\, dr
\\
=
\frac{1}{2\langle B\hat{\theta}, \hat{\theta}\rangle}\int_{0}^{+\infty}
\partial_r(\Phi(f_\theta) - \Phi(g_\theta))
\varrho_\theta\, dr
-
\frac{1}{2\langle B\hat{\theta}, \hat{\theta}\rangle}\int_{0}^{+\infty}
\varphi(f_\theta)\bigl(\partial_r f_\theta- \partial_r g_\theta\bigr)
\varrho_\theta\, dr.
\end{multline*}
We now integrate by parts in the first term above:
\begin{multline*}
\int_{0}^{+\infty}
\partial_r(\Phi(f_\theta) - \Phi(g_\theta))\varrho_\theta\, dr
\\
=
-\bigl(\Phi(f_\theta(0)) - \Phi(g_\theta(0))\bigr)\varrho_\theta(0)
+
\int_{0}^{+\infty}(\Phi(f_\theta) - \Phi(g_\theta))
\langle t+r\hat{\theta}, \hat{\theta}\rangle\varrho_\theta\, dr.
\end{multline*}
We note that
$|\Phi(f_\theta) - \Phi(g_\theta)|\le \|\varphi\|_\infty|f_\theta - g_\theta|$.
Therefore,
one gets
\begin{multline*}
\int_{0}^{+\infty}
[\varphi(f_\theta) - \varphi(g_\theta)]
r\varrho_\theta\, dr
\le
\frac{1}{2|\langle B\hat{\theta}, \hat{\theta}\rangle|}
\Bigl[(2\pi)^{-1}\sup_{x\in \mathbb{R}^2}\bigl[|f(x)-g(x)|e^{-\frac{1}{2}|x|^2}\bigr]
\\
+
\int_{0}^{+\infty}\bigl[|f_\theta - g_\theta|\cdot
|\langle t+r\hat{\theta}, \hat{\theta}\rangle| +
|\partial_r f_\theta- \partial_r g_\theta|\bigr]\varrho_\theta\, dr\Bigr].
\end{multline*}
We now note, that
\begin{multline*}
|f_\theta - g_\theta|\cdot
|\langle t+r\hat{\theta}, \hat{\theta}\rangle|\varrho_\theta
\le
(2\pi)^{-1}
\sup\limits_{x\in \mathbb{R}^2}\bigl[|f(x) - g(x)| |x| e^{-\frac{3}{8}|x|^2}\bigr]
e^{-\frac{1}{8}(r^2+2r\langle t, \hat{\theta}\rangle+|t|^2)}
\\
\le
2(2\pi)^{-1}
\sup\limits_{x\in \mathbb{R}^2}\bigl[|f(x) - g(x)| e^{-\frac{1}{4}|x|^2}\bigr]
e^{-\frac{1}{8}(r^2-2r|t|+|t|^2)},
\end{multline*}
where in the last estimate we use the bound $|x|e^{-\frac{1}{8}|x|^2}\le 2$.
We also note that
$$
|\partial_r f_\theta- \partial_r g_\theta| =
|\langle\nabla f(t+r\hat{\theta}) - \nabla g(t+r\hat{\theta}), \hat{\theta}\rangle|
\le |\nabla f(t+r\hat{\theta}) - \nabla g(t+r\hat{\theta})|.
$$
Thus,
\begin{multline*}
|\partial_r f_\theta- \partial_r g_\theta|\varrho_\theta
\le
(2\pi)^{-1}
\sup\limits_{x\in \mathbb{R}^2}\bigl[|\nabla f(x) - \nabla g(x)| e^{-\frac{1}{4}|x|^2}\bigr]
e^{-\frac{1}{4}(r^2+2r\langle t, \hat{\theta}\rangle+|t|^2)}
\\
\le
(2\pi)^{-1}
\sup\limits_{x\in \mathbb{R}^2}\bigl[|\nabla f(x) - \nabla g(x)| e^{-\frac{1}{4}|x|^2}\bigr]
e^{-\frac{1}{8}(r^2-2r|t|+|t|^2)},
\end{multline*}
Since
$$
\int_{0}^{+\infty}e^{-\frac{1}{8}(r^2-2r|t|+|t|^2)}\, dr
=
\int_{0}^{+\infty}e^{-\frac{1}{8}(r-|t|)^2}\, dr
\le
\int_{-\infty}^{+\infty}e^{-\frac{1}{8}(r-|t|)^2}\, dr
=\int_{-\infty}^{+\infty}e^{-\frac{1}{8}s^2}\, ds=
2\sqrt{2\pi},
$$
we get the bound
\begin{multline*}
\int_{0}^{+\infty}
[\varphi(f_\theta) - \varphi(g_\theta)]
r\varrho_\theta\, dr
\\
\le
\frac{1}{2\sqrt{2\pi}|\langle B\hat{\theta}, \hat{\theta}\rangle|}
\Bigl[5\sup_{x\in \mathbb{R}^2}\bigl[|f(x)-g(x)|e^{-\frac{1}{4}|x|^2}\bigr]
+
2\sup\limits_{x\in \mathbb{R}^2}\bigl[|\nabla f(x) - \nabla g(x)| e^{-\frac{1}{4}|x|^2}\bigr]\Bigr].
\end{multline*}
Without loss of generality we assume that
$\langle B\hat{\theta}, \hat{\theta}\rangle = s_1\cos^2\theta+s_2\sin^2\theta$
and $s_1$ and $s_2$ are of the same sign.
Thus, $|\langle B\hat{\theta}, \hat{\theta}\rangle| =
|s_1|\cos^2\theta+|s_2|\sin^2\theta$
and
\begin{multline*}
\int_{\mathbb{R}^2}[\varphi(f(x))-\varphi(g(x))]\, \gamma(dx)
\\
\le\frac{1}{\sqrt{2\pi}}
\Bigl[3\sup_{x\in \mathbb{R}^2}\bigl[|f(x)-g(x)|e^{-\frac{1}{4}|x|^2}\bigr]
+
\sup\limits_{x\in \mathbb{R}^2}\bigl[|\nabla f(x) - \nabla g(x)| e^{-\frac{1}{4}|x|^2}\bigr]\Bigr]
\int_0^{2\pi}\frac{d\theta}{|s_1|\cos^2\theta+|s_2|\sin^2\theta}.
\end{multline*}
We finally recall that
$$
\int_0^{2\pi}\frac{d\theta}{|s_1|\cos^2\theta+|s_2|\sin^2\theta}
=
4\int_{0}^{\pi/2}\frac{d{\rm tg}\theta}{|s_1|+|s_2|{\rm tg}^2\theta}
=
\frac{2\pi}{\sqrt{|s_1|\cdot|s_2|}}
$$
The lemma is proved.
\end{proof}

\begin{lemma}\label{lem2}
Let $X$ be the standard normal random vector
on $\mathbb{R}^2$.
Let $\ell_1, \ell_2$ be affine functions on $\mathbb{R}^2$, then
$$
\sup_{x\in\mathbb{R}^2}\bigl[\sqrt{|\ell_1(x)|^2+|\ell_2(x)|^2}e^{-\frac{1}{4}|x|^2}\bigr]
\le
\sqrt3\bigl(\mathbb{E}\bigl[|\ell_1(X)|^2+|\ell_2(X)|^2\bigr]\bigr)^{1/2}.
$$
Let $G$ be a second degree polynomial on $\mathbb{R}^2$, then
$$
\sup_{x\in\mathbb{R}^2}\bigl[|G(x)|e^{-\frac{1}{4}|x|^2}\bigr]
\le
6\sqrt2\bigl(\mathbb{E}\bigl[ |G(X)|^2\bigr]\bigr)^{1/2}.
$$
Moreover, for every $d\in \mathbb{N}$ there are positive numbers $c_1(d)$
and $c_2(d)$, dependent only on $d$, such that,
for every pair of polynomials $\ell_1, \ell_2$ of degree at most $d$,
one has
$$
\sup_{x\in\mathbb{R}^2}\bigl[\sqrt{|\ell_1(x)|^2+|\ell_2(x)|^2}e^{-\frac{1}{4}|x|^2}\bigr]
\le
c_1(d)\bigl(\mathbb{E}\bigl[|\ell_1(X)|^2+|\ell_2(X)|^2\bigr]\bigr)^{1/2}.
$$
and, for every polynomial $G$ of degree at most $d$, one has
$$
\sup_{x\in\mathbb{R}^2}\bigl[|G(x)|e^{-\frac{1}{4}|x|^2}\bigr]
\le
c_2(d)\bigl(\mathbb{E}\bigl[ |G(X)|^2\bigr]\bigr)^{1/2}.
$$
\end{lemma}

\begin{proof}
Let $\ell_j(x)=a_{1,j}x_1+a_{2,j}x_2+b_j$.
Then
$$
\mathbb{E}\bigl[|\ell_1(X)|^2+|\ell_2(X)|^2\bigr]
=
\sum_{j=1}^2a_{1,j}^2+a_{2,j}^2+b_j^2.
$$
We note that
$\sup\limits_{t\in \mathbb{R}}\bigl[|t|^me^{-\frac{1}{2}t^2}\bigr]=(m)^{\frac{m}{2}}e^{-\frac{m}{2}}$
for any $m\in \mathbb{N}$. Thus,
$\sup\limits_{t\in \mathbb{R}}\bigl[|t|e^{-\frac{1}{2}t^2}\bigr]=e^{-1/2}\le1$ and
$\sup\limits_{t\in \mathbb{R}}\bigl[t^2e^{-\frac{1}{2}t^2}\bigr]=2e^{-1}\le 1$.
Therefore,
\begin{multline*}
\bigl(\sum_{j=1}^2|\ell_j(x)|^2\bigr)e^{-\frac{1}{2}|x|^2}
\le
\sum_{j=1}^2\bigl(a_{1,j}^2+a_{2,j}^2+b_j^2+2|a_{1,j}||a_{2,j}|
+2|b_j||a_{1,j}|+2|b_j||a_{2,j}|\bigr)
\\
\le
3\sum_{j=1}^2(a_{1,j}^2+a_{2,j}^2+b_j^2).
\end{multline*}
The first estimate of the lemma is proved.

Let $G(x)=a_{11}x_1^2+2a_{12}x_1x_2+a_{22}x_2^2 + b_1x_1+b_2x_2+c$,
then
\begin{multline*}
\mathbb{E}\bigl[ |G(X)|^2\bigr]
=
3(a_{11}^2+a_{22}^2)+4a_{12}^2+b_1^2+b_2^2+c^2+2a_{11}a_{22}+2c(a_{11}+a_{22})
\\
=
\tfrac{3}{2}(a_{11}^2+a_{22}^2) +4a_{12}^2+b_1^2+b_2^2 + \tfrac{1}{3}c^2 -a_{11}a_{22}+
\bigl(\tfrac{\sqrt2}{\sqrt3}c+\tfrac{\sqrt3}{\sqrt2}(a_{11}+a_{22})\bigr)^2
\\
\ge
a_{11}^2+a_{22}^2 +4a_{12}^2+b_1^2+b_2^2 + \tfrac{1}{3}c^2
\ge\tfrac{1}{3}(a_{11}^2+a_{22}^2 +a_{12}^2+b_1^2+b_2^2 + c^2).
\end{multline*}
We note that
$\sup\limits_{t\in \mathbb{R}}\bigl[|t|^me^{-\frac{1}{4}t^2}\bigr]=(2m)^{\frac{m}{2}}e^{-\frac{m}{2}}$
for any $m\in \mathbb{N}$. Thus,
$\sup\limits_{t}\bigl[|t|e^{-t^2/4}\bigr]=\sqrt2e^{-1/2}\le1$ and
$\sup\limits_{t}\bigl[t^2e^{-t^2/4}\bigr]=4e^{-1}\le 2$.
Therefore,
\begin{multline*}
\sup_{x\in\mathbb{R}^2}\bigl[|G(x)|e^{-\frac{|x|^2}{4}}\bigr]
\le
2|a_{11}|+2|a_{12}|+2|a_{22}|+|b_1|+|b_2|+|c|
\\
\le
2(|a_{11}|+|a_{12}|+|a_{22}|+|b_1|+|b_2|+|c|)
\le
2\sqrt6\sqrt{a_{11}^2+a_{12}^2 +a_{22}^2 +b_1^2+b_2^2 + c^2}.
\end{multline*}
Thus,
$$
\sup_{x\in\mathbb{R}^2}\bigl[|G(x)|e^{-\frac{|x|^2}{4}}\bigr]
\le
6\sqrt2 \bigl(\mathbb{E}\bigl[ |G(X)|^2\bigr]\bigr)^{1/2}.
$$
The second claim of the lemma is proved.

Now, if we consider a linear space $L$ of all
mappings $(\ell_1, \ell_2)$, where $\ell_1$ and $\ell_2$
are polynomials of degree at most $d$,
then the both expressions
$$(\ell_1, \ell_2)\mapsto
\sup_{x\in\mathbb{R}^2}\bigl[\sqrt{|\ell_1(x)|^2+|\ell_2(x)|^2}e^{-\frac{1}{4}|x|^2}\bigr]
$$
and
$$(\ell_1, \ell_2)\mapsto
\bigl(\mathbb{E}\bigl[|\ell_1(X)|^2+|\ell_2(X)|^2\bigr]\bigr)^{1/2}.
$$
define norms on the space $L$.
Every pair of norms on the finite dimensional space is equivalent
which implies the existence of the number $c_1(d)$.
By the similar reasoning one gets the existence of $c_2(d)$.
The lemma is proved.
\end{proof}

\begin{lemma}\label{lem3}
Let $A\colon\mathbb{R}^n\to\mathbb{R}^n$ be a self-adjoint operator,
let $a\in\mathbb{R}^n$ be a fixed vector,
and let $\alpha\in\mathbb{R}$. Let $Z=(Z_1,\ldots, Z_n)$
be the standard normal random vector, i.e. its components $Z_j$
are independent random variables with $\mathcal{N}(0, 1)$ distribution.
Then
$$
\mathbb{E}\bigl[\langle AZ, Z\rangle + \langle a, Z\rangle+\alpha\bigr]
=
{\rm tr}A + \alpha
$$
and
$$
\mathbb{E}\bigl[|\langle AZ, Z\rangle + \langle a, Z\rangle+\alpha|^2\bigr]
=
2\|A\|_{HS}^2+({\rm tr}A+\alpha)^2+|a|^2.
$$
\end{lemma}

\begin{proof}
Without loss of generality
we can assume that $Ax = (\lambda_1x_1, \ldots, \lambda_nx_n)$.
Then,
$$
\mathbb{E}\bigl[\langle AZ, Z\rangle + \langle a, Z\rangle+\alpha\bigr]
=
\mathbb{E}\bigl[\sum_{j=1}^n \lambda_jZ_j^2\bigr] + \alpha
\\
=
\sum_{j=1}^n \lambda_j+ \alpha
={\rm tr}A + \alpha.
$$

For the second part
we have
\begin{multline*}
\mathbb{E}\bigl[|\langle AZ, Z\rangle + \langle a, Z\rangle+\alpha|^2\bigr]
\\
=
\mathbb{E}\bigl[|\langle AZ, Z\rangle|^2+|\langle a, Z\rangle|^2 +\alpha^2
+ 2\langle AZ, Z\rangle\langle a, Z\rangle\ +
2\alpha\langle a, Z\rangle + 2\alpha\langle AZ, Z\rangle\bigr]
\\
=
\mathbb{E}\Bigl|\sum_{j=1}^n\lambda_jZ_j^2\Bigr|^2
+2\alpha\mathbb{E}\bigl[\sum_{j=1}^n\lambda_jZ_j^2\bigr] +\alpha^2  + |a|^2
\\
=
3\sum_{j=1}^n\lambda_j^2 + 2\sum_{1\le j< k\le n}\lambda_j\lambda_k +
2\alpha\sum_{j=1}^n\lambda_j +\alpha^2 + |a|^2
\\
=
2\sum_{j=1}^n\lambda_j^2+\Bigl(\sum_{j=1}^n\lambda_j\Bigr)^2 +
2\alpha\sum_{j=1}^n\lambda_j +\alpha^2 + |a|^2
=
2\|A\|_{HS}^2+({\rm tr}A+\alpha)^2+|a|^2.
\end{multline*}
The lemma is proved.
\end{proof}

In particular, we get the following corollary.

\begin{corollary}\label{cor1}
Let $f$ and $g$ be two second degree polynomials in $\mathbb{R}^n$
and let $Z$ be the standard $n$-dimensional normal random vector.
Then
$$
\|\nabla f(Z) - \nabla g(Z)\|_2\le \sqrt{2}\|f(Z) - g(Z)\|_2.
$$
\end{corollary}

\begin{proof} One has
$f(x) = \langle Ax, x\rangle + \langle a, x\rangle + \alpha$
and $g(x) = \langle Bx, x\rangle +\langle b, x\rangle +\beta$
for some self-adjoint operators $A$ and $B$, for some $a,b\in \mathbb{R}^n$, and
for some
$\alpha, \beta\in \mathbb{R}$.
Then $\nabla f(x) = 2Ax+a$ and $\nabla g(x) = 2Bx +b$.
By Lemma \ref{lem3} we have
\begin{multline*}
\|\nabla f(Z) - \nabla g(Z)\|_2^2
= \mathbb{E}\bigl[|2(A-B)Z+(a-b)|^2\bigr]
\\
=
\mathbb{E}\bigl[4\langle(A-B)^2Z, Z\rangle + 4\langle(A-B)(a-b), Z\rangle + |a-b|^2\bigr]
\\
=
4{\rm tr}(A-B)^2 + |a-b|^2
=4\|A-B\|_{HS}^2+ |a-b|^2.
\end{multline*}
On the other hand, by Lemma \ref{lem3},
$$
\|f(Z) - g(Z)\|_2^2 = 2\|A-B\|_{HS}^2 +
({\rm tr}A - {\rm tr}B+\alpha-\beta)^2+|a-b|^2\ge 2\|A-B\|_{HS}^2 + \frac{1}{2}|a-b|^2.
$$
Thus,
$$
\|\nabla f(Z) - \nabla g(Z)\|_2^2\le 2 \|f(Z) - g(Z)\|_2^2
$$
and the corollary is proved.
\end{proof}

\begin{remark}{\rm
We note that the equivalence
of the $L^2$-norm
and the Sobolev norm of the Gaussian Sobolev space $W^{1,2}$
on the space of all polynomials of degree at most~$d$
(see \cite[Corollary 5.5.5 and Theorem 5.7.2]{Gaus})
implies that,
for each positive integer $d$ there is a number $c(d)$,
dependent only on $d$, such that, for any
polynomial $f$ of degree at most $d$ and for any second degree polynomial $g$,
one has
\begin{equation}\label{nabla-bound}
\|\nabla f(Z) - \nabla g(Z)\|_2\le c(d)\|f(Z) - g(Z)\|_2.
\end{equation}
The previous lemma specifies the constant $c(d)$ in the case when
$f$ is also of the second degree.
}
\end{remark}

We are now ready to prove Theorem \ref{T1}.

{\bf Proof of Theorem \ref{T1}.}
Due to rotation invariance of the standard Gaussian measure,
without loss of generality we can assume that
$L=\{(x_1, x_2, 0, \ldots, 0)\colon x_1, x_2\in \mathbb{R}\}$.
Consider the random vector $X=(Z_1, Z_2)$
and let $\varphi\in C_0^\infty(\mathbb{R})$ be any function such that $\|\varphi\|_\infty\le 1$.
By Lemma \ref{lem1} and Lemma \ref{lem2} one has
\begin{multline*}
\mathbb{E}
\bigl[\varphi\bigl(f(X, x_3, \ldots, x_n)\bigr)-\varphi\bigl(g(X, x_3, \ldots, x_n)\bigr)\bigr]
\\
\le
\frac{\sqrt{2\pi}}{\sqrt{|s_1|\cdot|s_2|}}
\Bigl[3\sup\limits_{x_1, x_2\in \mathbb{R}}
\bigl[|f(x_1, x_2, x_3,\ldots, x_n)-g(x_1, x_2, x_3, \ldots,x_n)|e^{-\frac{1}{4}(x_1^2+x_2^2)}\bigr]
\\
+
\sup\limits_{x_1, x_2\in \mathbb{R}}
\bigl[|\nabla_{x_1, x_2} f(x_1, x_2, x_3,\ldots, x_n) -
\nabla_{x_1, x_2} g(x_1, x_2, x_3,\ldots, x_n)| e^{-\frac{1}{4}(x_1^2+x_2^2)}\bigr]\Bigr]
\\
\le
\frac{\sqrt{2\pi}}{\sqrt{|s_1|\cdot|s_2|}}
\Bigl[3c_2(d)\bigl(\mathbb{E}
|f(X, x_3,\ldots, x_n) - g(X, x_3,\ldots, x_n)|^2\bigr)^{1/2}
\\
+c_1(d)\bigl(\mathbb{E}
|\nabla_{x_1, x_2}\bigl(f(X, x_3,\ldots, x_n) - g(X, x_3,\ldots, x_n)\bigr)|^2\bigr)^{1/2}
\Bigr]
\\
\le
\frac{\sqrt{2\pi}}{\sqrt{|s_1|\cdot|s_2|}}
\Bigl[\bigl(3c_2(d)+c_1(d)c(d)\bigr)\bigl(\mathbb{E}
|f(X, x_3,\ldots, x_n) - g(X, x_3,\ldots, x_n)|^2\bigr)^{1/2}\Bigr],
\end{multline*}
where in the last estimate we use inequality \eqref{nabla-bound}.
Finally,
\begin{multline*}
\mathbb{E}\bigl[\varphi\bigl(f(Z)\bigr)\bigr] - \mathbb{E}\bigl[\varphi\bigl(g(Z)\bigr)\bigr]
=
\mathbb{E}_{Z_3,\ldots, Z_n}\mathbb{E}_X
\bigl[\varphi\bigl(f(X, Z_3, \ldots, Z_n)\bigr)-\varphi\bigl(g(X, Z_3, \ldots, Z_n)\bigr)\bigr]
\\
\le
\frac{C(d)}{\sqrt{|s_1|\cdot|s_2|}}\mathbb{E}_{Z_3,\ldots, Z_n}
\bigl(\mathbb{E}_X
|f(X, Z_3,\ldots, Z_n) - g(X, Z_3,\ldots, Z_n)|^2\bigr)^{1/2}
\\
\le
\frac{C(d)}{\sqrt{|s_1|\cdot|s_2|}}
\bigl(\mathbb{E}
|f(Z) - g(Z)|^2\bigr)^{1/2}
\end{multline*}
where $C(d) = \sqrt{2\pi}\bigl(3c_2(d)+c_1(d)c(d)\bigr)$,
which implies the first announced bound.

To get the constant in the case when $f$
is of degree at most two we note that
in this case one can take $c_1(d)=\sqrt{3}$, $c_2(d)=6\sqrt{2}$,
and $c(d)=\sqrt{2}$, which follows
from Lemma~\ref{lem2} and Corollary~\ref{cor1}.
Since $\sqrt{2\pi}(18\sqrt{2}+\sqrt{6})\le80$,
we get the second claim of the theorem.
The theorem is proved.
\qed

\begin{remark}\label{rem2}{\rm
We note that by Lemma \ref{lem3}
for polynomials
$f(x) = \langle Ax, x\rangle + \langle a, x\rangle + \alpha$
and $g(x) = \langle Bx, x\rangle +\langle b, x\rangle +\beta$,
where $A$ and $B$ are self-adjoint operators,
$a,b\in \mathbb{R}^n$, and
$\alpha, \beta\in \mathbb{R}$,
we have
\begin{multline*}
\|f(Z)-g(Z)\|_2=\bigl(2\|A-B\|_{HS}^2+({\rm tr}A - {\rm tr}B + \alpha-\beta)^2+|a-b|^2\bigr)^{1/2}
\\
\le
2(\|A-B\|_{HS} + |{\rm tr}A - {\rm tr}B + \alpha-\beta|+|a-b|).
\end{multline*}
Thus, under the same assumptions as in Theorem \ref{T1},
we have
$$
d_{\rm TV}\bigl(f(Z), g(Z)\bigr)\le \frac{160}{\sqrt{|\lambda_1|\cdot|\lambda_2|}}
\bigl(\|A-B\|_{HS} + |{\rm tr}A - {\rm tr}B + \alpha-\beta|+|a-b|\bigr).
$$
In particular, when $\mathbb{E}f(Z) = \mathbb{E}g(Z)$, we get
$|{\rm tr}A - {\rm tr}B + \alpha-\beta|=0$ and the bound is
$$
d_{\rm TV}\bigl(f(Z), g(Z)\bigr)\le \frac{160}{\sqrt{|s_1|\cdot|s_2|}}
\bigl(\|A-B\|_{HS}+|a-b|\bigr).
$$
}
\end{remark}

\section{Applications.}

We firstly obtain a sharper version of the result from \cite{Zin}.
Recall that
$\Lambda(B)$ denotes the set of all eigenvalues of
a self adjoint linear operator $B$, counting multiplicities.

\begin{corollary}\label{cor2}
Let $Z$ be the standard $n$-dimensional normal random vector and
let $d\in \mathbb{N}$. There is a number $C(d)$, dependent only on $d$,
such that, for any
$g(x) = \langle Bx, x\rangle +\langle b, x\rangle +\beta$,
for which the cardinality of the set $\{\lambda\in\Lambda(B)\colon \lambda\ne0\}$
is at least three, and for any polynomial
$f$ of degree at most $d$, one has
$$
d_{\rm TV}\bigl(f(Z), g(Z)\bigr)\le \frac{C(d)}{\sqrt{|s_1|\cdot|s_2|}}\|f(Z)-g(Z)\|_2
$$
where $s_1$ and $s_2$ are any two eigenvalues from $\Lambda(B)$
of the same sign.
When $f$ is of degree at most $2$, one can take $C(2) = 80$.
In particular, if the eigenvalues from $\Lambda(B)$ is enumerated
such that $|\lambda_1|\ge |\lambda_2|\ge|\lambda_3|\ge\ldots$, then
$$
d_{\rm TV}\bigl(f(Z), g(Z)\bigr)\le \frac{C(d)}{\sqrt{|\lambda_2|\cdot|\lambda_3|}}\|f(Z)-g(Z)\|_2.
$$
When $f$ is of degree at most two, one can take $C(2)=80$.
\end{corollary}

\begin{proof}
We can apply Theorem \ref{T1} to the eigenspace $L$
corresponding to
$s_1$ and $s_2$.
Moreover, if we take the biggest eigenvalues of the same sign, then always
$|s_1|\cdot|s_2|\ge |\lambda_2|\cdot|\lambda_3|$.
\end{proof}

\begin{remark}
{\rm
If we now take $f=g+h$, $h\in \mathbb{R}$,
we obtain that the density of the random variable
$g(x) = \langle Bx, x\rangle +\langle b, x\rangle +\beta$
belongs to the space of functions of bounded variation
provided
that the cardinality of the set $\{\lambda\in\Lambda(B)\colon \lambda\ne0\}$
is at least three.
}
\end{remark}

Secondly, we provide the following generalization of the bounds from \cite{GNSU},
announced in the introduction.

\begin{corollary}\label{cor3}
Let $X$ and $Y$ be two centered $n$-dimensional normal random vectors with the covariance
matrixes $\Sigma_X$ and $\Sigma_Y$ respectively and let $a,b\in \mathbb{R}^n$.
Then
$$
d_{\rm TV}(|X-a|, |Y-b|)
\le
\frac{160}{\sqrt{\lambda_{1X}\cdot\lambda_{2X}}}
\bigl(\|\Sigma_X-\Sigma_Y\|_{HS}+
|{\rm tr}\Sigma_X-{\rm tr}\Sigma_Y|+\bigl||a|^2-|b|^2\bigr|+
|\Sigma_X^{1/2}a - \Sigma_Y^{1/2}b|\bigr).
$$
\end{corollary}

\begin{proof}
Without loss of generality
we can assume that $X=\Sigma_X^{1/2}Z$ and $Y=\Sigma_Y^{1/2}Z$,
where $Z$ is the standard normal $n$-dimensional random vector.
Then
$$
|X-a|^2 = \langle \Sigma_X^{1/2}Z-a, \Sigma_X^{1/2}Z-a\rangle
= \langle \Sigma_XZ, Z\rangle - 2\langle \Sigma_X^{1/2}a, Z\rangle + |a|^2=g(Z)$$
and
$$
|Y-b|^2=\langle \Sigma_Y^{1/2}Z - b, \Sigma_Y^{1/2}Z -b\rangle
=\langle \Sigma_YZ, Z\rangle - 2\langle \Sigma_Y^{1/2}b, Z\rangle + |b|^2=f(Z).
$$
We now apply Theorem \ref{T1} to these two polynomials
$f(Z)$ and $g(Z)$:
\begin{multline*}
d_{\rm TV}(|X-a|, |Y-b|)
= d_{\rm TV}(|X-a|^2, |Y-b|^2)=d_{\rm TV}(f(Z), g(Z))
\\
\le \frac{80}{\sqrt{\lambda_{1X}\cdot\lambda_{2X}}}\|f(Z)-g(Z)\|_2
\end{multline*}
where $\lambda_{1X}$ and $\lambda_{2X}$ are the first two eigenvalues of the matrix $\Sigma_X$.
By Lemma \ref{lem3}
$$
\|f(Z)-g(Z)\|_2^2 = 2\|\Sigma_X-\Sigma_Y\|_{HS}^2+
\bigl({\rm tr}\Sigma_X-{\rm tr}\Sigma_Y+|a|^2-|b|^2\bigr)^2+
4\bigl|\Sigma_X^{1/2}a-\Sigma_Y^{1/2}b\bigr|^2
$$
implying the announced bound.
The corollary is proved.
\end{proof}

\section*{Acknowledgment}

The author is a Young Russian Mathematics award winner and would like to thank its sponsors and jury.

The article was prepared within the framework of the HSE University Basic Research Program
and funded by the Russian Academic Excellence Project '5-100'.

This research was also supported by the Russian Foundation for Basic Research Grant 20-01-00432 
and the Moscow Center of Fundamental and Applied Mathematics

\end{document}